\documentclass[11pt]{amsart}

\usepackage{amsmath, amscd}
\usepackage{enumerate}
\usepackage{amsthm}
\usepackage{amssymb}
\usepackage[all]{xy}
\usepackage{verbatim}

\numberwithin{equation}{section}

\newtheorem{thm}{Theorem}[section]

\newtheorem{lem}[thm]{Lemma}

\newtheorem{cor}[thm]{Corollary}

\newtheorem{prop}[thm]{Proposition}

\newtheoremstyle{cond}{}{}{}{}{\bfseries}{.}{ }{#1 #3}
\theoremstyle{cond}

\newtheoremstyle{mystyle}{}{}{}{}{\sffamily\scshape\bfseries}{}{\newline}{\thmnote{#3}}
\theoremstyle{mystyle}

\newtheoremstyle{mynum}{}{}{\itshape}{}{\bfseries}{.}{ }{#1 #3}
\theoremstyle{mynum}

\numberwithin{equation}{section}

\theoremstyle{definition}
\newtheorem{ex}[thm]{Example}

\newcommand{\bZ}{\mathbb{Z}}
\newcommand{\ZZ}{\mathbb{Z}}

\newcommand{\bR}{\mathbb{R}}
\newcommand{\bN}{\mathbb{N}}
\newcommand{\bC}{\mathbb{C}}

\newcommand{\PP}{\mathbb{P}}

\newcommand{\cA}{\mathcal{A}}

\newcommand{\cM}{\mathcal{M}}

\newcommand{\cX}{\mathcal{X}}

\newcommand{\cF}{\mathcal{F}}

\newcommand{\cG}{\mathcal{G}}

\newcommand{\cO}{\mathcal{O}}

\newcommand{\XX}{\mathbb{X}}

\newcommand{\Hom}{{\rm Hom}}

\newcommand{\supp}{{\rm Supp}}

\newcommand{\la}{\langle}
\newcommand{\ra}{\rangle}

\newcommand{\coh}{{\sf coh}}
\newcommand{\Qcoh}{{\sf Qcoh}}
\newcommand{\Db}{{\sf D}^b}

\newcommand{\Zol}{\overline{Z}}
\newcommand{\Phiol}{\overline{\Phi}}

\def\gr{\operatorname {\sf gr}}
\def\Gr{\operatorname {\sf Gr}}
\def\fdim{\operatorname {\sf fdim}}
\def\qgr{\operatorname {\sf qgr}}

\def\QGr{\operatorname {\sf QGr}}

\def\Ext{{\rm Ext}}

\mathchardef\mhyphen="2D

\newcommand{\set}[1]{\left\{#1\right\}}

\newcommand{\ol}[1]{\overline{#1}}

\renewcommand{\l}{\ell}

\begin{document}

\title[Regular algebras of global dimension two]{The Grothendieck group of non-commutative non-noetherian analogues of $\PP^1$ and 
\\
Regular algebras of global dimension two}

\date{\today}

\author{Gautam Sisodia and S. Paul Smith}

\address{ Department of Mathematics, Box 354350, Univ.
Washington, Seattle, WA 98195}

\email{gautas@math.washington.edu, smith@math.washington.edu}


\begin{abstract}
Let $V$ be a finite-dimensional positively-graded vector space. Let $b \in V \otimes V$ be a homogeneous element
whose rank  is  $\dim(V)$. Let $A=TV/(b)$, the quotient of the tensor algebra $TV$ modulo the 2-sided ideal 
generated by $b$.  Let $\gr(A)$ be the category of finitely presented graded left $A$-modules and $\fdim(A)$ its full subcategory 
of finite dimensional  modules. Let $\qgr(A)$ be the quotient category $\gr(A)/\fdim(A)$. 
We compute the Grothendieck group $K_0(\qgr(A))$. In particular, if the reciprocal of the 
Hilbert series of $A$, which is a polynomial, is irreducible, then $K_0(\qgr(A)) \cong \bZ[\theta] \subset \bR$ as ordered abelian
groups where $\theta$ is the smallest positive real root of that polynomial. 
When $\dim_k(V)=2$, $\qgr(A)$ is equivalent to the category of coherent sheaves on the projective line, $\PP^1$, or a 
stacky $\PP^1$ if $V$ is not concentrated in
degree 1. If $\dim_k(V) \ge 3$, results of Piontkovskii and Minamoto 
suggest that $\qgr(A)$ behaves as if it is the category of ``coherent sheaves'' on a non-commutative, 
non-noetherian, analogue of $\PP^1$. 
\end{abstract}

\maketitle

\section{Introduction}

\subsection{}
Let $k$ be a field. Let $A=TV/(b)$ where $V$ is a  finite-dimensional positively-graded $k$-vector space
and $b \in V \otimes V$ a homogeneous element of rank  $\dim(V)$. Zhang \cite{zh} showed that up to equivalence the category $\gr(A)$ depends only
on $V$ as a graded vector space, not on $b$.

This paper is motivated by non-commutative algebraic geometry. As we explain in \S\ref{sect.ncag}, 
results of Piontkovskii and Minamoto suggest that the category $\qgr(A)$ behaves as if it is the category of 
``coherent sheaves'' on a non-commutative, non-noetherian if $\dim_k(V) \ge 3$, analogue of the projective line, $\PP^1$. 
With this perspective we are computing the Grothendieck groups of ``coherent sheaves'' on these non-commutative analogues 
of $\PP^1$. 

The description of $K_0(\qgr(A))$ as an {\it ordered} abelian group says that 
$\alpha \in \bZ[\theta]\cong K_0(\qgr(A))$ is equal to $[\cF]$ for some $\cF \in \qgr(A)$ if and only if $\alpha \ge 0$. 

The result reminds us of the fact that the $K_0$ of an irrational rotation algebra 
$\cA_\theta$ is isomorphic to $\bZ[\theta] \subseteq \bR_{\ge 0}$ as an ordered abelian group \cite{PV}, \cite{R}.
We do not know any connection between $\qgr(A)$ and $\cA_\theta$.

\subsection{The case $\dim_k(V)=2$} 
\label{ssect.dim=2}
When $V=kx_0 \oplus kx_1$ with $\deg(x_0)=\deg(x_1)=1$ and $b=x_0x_1-x_1x_0$, 
$A$ is the polynomial ring $k[x_0,x_1]$ and  the category $\qgr(A)$ of finitely generated graded $A$-modules modulo the full subcategory of finite dimensional graded $A$-modules  is equivalent to the category
$\coh(\PP^1)$ of coherent sheaves on $\PP^1$. If $\deg(x_0)=1$ and $\deg(x_1)=m>1$ and $b=x_0x_1-x_1x_0$, then $\qgr(A)$
is equivalent to $\coh[\PP^1/\bZ_m]$, the coherent sheaves on the stacky $\PP^1$ with a single stacky point  
isomorphic to $B\bZ_m$. The $K_0$ of this stack, and more general toric DM stacks, is computed in \cite{ps}.

\subsection{}
\label{sect.ncag} 
 In \cite{piont}, Piontkovskii shows that $A=TV/(b)$ behaves like a homogeneous coordinate ring of a non-commutative (non-noetherian if $n \ge 3$) analogue of the projective line. 
He proves that $A$ is graded coherent and hence that the category $\qgr(A)$ of finitely
presented graded $A$-modules modulo the full subcategory of finite dimensional graded $A$-modules is abelian category. 
He shows that $\qgr(A)$ is like $\coh(\PP^1)$ in so far as it has cohomological dimension 1, $\Ext$ groups have finite dimension, and satisfies Serre duality. Explicitly,  if $\cF,\cG \in \qgr(A)$, then $\Ext_{\qgr(A)}^2(\cF,\cG)=0$, $\dim_k\Ext_{\qgr(A)}^*(\cF,\cG) <\infty$,  and $\Hom_{\qgr(A)}(\cF,\cG) \cong \Ext_{\qgr(A)}^1(\cG,\cF(-n-1))^*$.  

In \cite{M}, Minamoto gives additional evidence that $\qgr(A)$ is like $\coh(\PP^1)$ by proving an analogue of the
well-known equivalence, $\Db(\Qcoh(\PP^1_k)) \equiv \Db(kQ_{2})$, of bounded derived categories where $kQ_{2}$ is the 
path algebra of the quiver with two vertices and $2$ arrows from the first vertex to the second. Minamoto shows that
$\Db(\Qcoh(A)) \equiv \Db(kQ_{n+1})$ where $Q_{n+1}$ is the path algebra of the quiver with two vertices and $n+1$ arrows
from the first vertex to the second. 

\subsection{}
The categories $\qgr(A)$ also appear in a paper of Kontsevich and Rosenberg \cite{KR}. 
They  propose a non-commutative analogue of $\PP^n_k$, $N\PP^n_k$, and say that  $\Db(\Qcoh(N\PP^n_k)) \equiv \Db(kQ_{n+1})$. By Minamoto's result, 
$\Qcoh(N\PP^n_k)$ is derived equivalent to $\QGr\big(k\langle x_0,\ldots,x_n\rangle/(x_0^2+\cdots+x_n^2)\big)$. 

\subsection{}
The Grothendieck groups of various {\it noetherian} non-commutative varieties are already known
and have proved useful. For various noetherian analogues $\XX$ of $\PP^n$, 
$K_0(\XX) \cong K_0(\PP^n) \cong \bZ^{n+1}$ \cite{mosmi}. It is shown there that in those cases one can use $K_0(\XX)$
to develop a rudimentary but useful intersection theory for $\XX$. It is shown in 
\cite{MS2} that $K_0$ for certain non-commutative ruled surfaces behaves ``just like'' the commutative case. 
Although the geometry of smooth non-commutative quadric surfaces in non-commutative analogues of $\PP^3$ is very much like 
that  for their commutative counterparts, $K_0$ exhibits an interesting non-commutative phenomenon such as the existence of
simple objects that behave like curves with self-intersection $-2$ (see \cite{SVdB} for details).  

\subsection{Regular algebras}
\label{ssect.regular}
Let $k$ be a field and $A=\bigoplus_{n \ge 0} A_n$ an $\bN$-graded $k$-algebra such that $A_0=k$. The left and right 
global dimensions of $A$ are the same and equal the projective dimension of the $A$-module $k:=A/A_{\ge 1}$. We say
$A$ is  {\sf regular} if it has finite global dimension, $n$ say, and 
$$
\Ext_A^j(k,A) \cong \begin{cases}
k & \text{if $j=n$}
\\
0 & \text{if $j \ne n$.}
\end{cases}
$$
Zhang \cite[Theorem 0.1]{zh}  proved that  $A$  is regular of global dimension 2  if and only if 
it is isomorphic to some 
\begin{equation}
\label{defn.A}
A:=\frac{k\la x_1, \ldots, x_g \ra}{(b)}
\end{equation}
where $g \ge 2$, the $x_i$'s can be labelled so that $\deg(x_i)+\deg(x_{g+1-i})=:d$ is the same for all $i$, and $\sigma$ is a graded $k$-algebra automorphism of the free algebra 
$k \langle x_1,\ldots,x_g \ra$, and  $b=\sum_{i = 1}^g x_i\sigma(x_{g+1-i})$.

\subsection{}
From now on $A$ denotes the algebra in (\ref{defn.A}) where the degrees of the generators and the relation $b$ 
have the properties stated after (\ref{defn.A}). 

\subsection{}
\label{ssect.coherent}
Let $\Gr(A)$ denote the category of graded left $A$-modules and $\gr(A)$ its full subcategory of finitely presented modules. Because $A$ is defined by a single homogeneous relation $\gr(A)$ is an abelian category  \cite[Theorem 1.2]{piont}. 
Let $\fdim(A)$ be the full subcategory of finite 
dimensional graded left $A$-modules and write $\qgr(A)$ for the quotient category $\gr(A)/\fdim(A)$, and $\pi^*: \gr(A) \to \qgr(A)$ for the quotient functor. 

If $M$ is a graded $A$-module, $M(1)$ denotes the graded module that is $M$ as an abelian group with grading $M(1)_n=M_{1+n}$ and the same action of $A$.
Since $\fdim(A)$ is stable under the functor $M \mapsto M(1)$ there is an induced functor $\pi^*M \mapsto \pi^*(M(1)) = : (\pi^*M)(1)$
on the quotient category $\qgr(A)$.

\subsubsection{}
We follow the convention in algebraic geometry where, for a smooth projective variety $X$, 
$K_0(\coh(X))$ denotes the free abelian group generated by coherent $\cO_X$-modules modulo the relations $[\cM_2]=[\cM_1]+[\cM_3]$ for every exact sequence $0 \to \cM_1 \to \cM_2 \to \cM_3 \to 0$. 
In this paper, 
$K_0(\qgr(A))$ denotes the free abelian group generated by the objects in $\qgr(A)$ modulo the 
``same'' relations.
The order structure on $K_0(\qgr A)$ is defined by $K_0(\qgr(A))_{\geq 0} := \set{[M] \ |\ M \in \qgr(A)}$.

 \subsection{}
 Because $A$ is regular of global dimension two the minimal projective resolution of ${}_Ak$ is  
$$
\xymatrix{
0 \ar[r] & A(-d) \ar[r]^-\alpha & \bigoplus_{i = 1}^g A(-\deg(x_i)) \ar[r]^-\beta & A \ar[r] & k \ar[r] & 0
}
$$
where $d=\deg(x_i)+\deg(x_{g+1-i})$, $\alpha$ is right multiplication by $(x_g, \ldots, x_1)$,
 and $\beta$ is right multiplication by $(\sigma(x_1), \sigma(x_2), \ldots,\sigma(x_g))^{\sf T}$. 
 The Hilbert series for $A$ is therefore
 $$
 H_A(t):= \sum_{n=0}^\infty \dim_k(A_n)t^n = \frac{1}{f(t)}  
$$
where
\begin{equation}
\label{defn.f}
f(t):= t^d - \sum_{i = 1}^g t^{\deg(x_i)} + 1.
\end{equation}

\subsection{}
The minimal projective resolution of a module $M \in \gr(A)$ has  the form 
$$
0 \to \bigoplus A(-i)^{e_i} \to \bigoplus A(-i)^{c_i} \to \bigoplus A(-i)^{b_i} \to M \to 0
$$
where the sums are finite. Since $H_{A(-i)}(t) = t^iH_{A}(t)$,  the Hilbert series for $M$, $H_M(t)$, is $H_A(t)q_M(t)$ where 
$$
q_M(t):=\sum_{i \ge 0} (b_i-c_i+e_i)t^i \; \in \; \bZ[t^{\pm 1}].
$$

\subsection{The main result and remarks}
Under the hypotheses in the next theorem, we make $\bZ[t^{\pm 1}]/(f)$ an ordered abelian group by defining 
$$
\left( \frac{\bZ[t, t^{- 1}]}{(f)} \right)_{\geq 0} := \set{\ol{p} \ |\ p(\theta) > 0} \cup \set{0}
$$
where $\ol{p}$ denotes the image of the Laurent polynomial $p$ in $\bZ[t^{\pm 1}]/(f)$.

\begin{thm}
\label{thm.main}
Let $A$ be the algebra in  (\ref{defn.A}) and assume that the greatest common divisor of the 
degrees of its generators $x_i$ is 1. Let $f(t)$ be the polynomial in (\ref{defn.f}) and let $\theta$ be its smallest positive real root.
Let $\bZ[\theta] \subset \bR$ be the $\bZ$-subalgebra generated by $\theta$ viewed as an ordered abelian subgroup  of $\bR$. 
The Grothendieck group $K_0(\qgr(A))$ is isomorphic as an ordered abelian group to
$$
\frac{\bZ[t, t^{- 1}]}{(f)}
$$
via the map $[\pi^* M] \mapsto  \ol{q_M(t)}$. If $f$ is irreducible, $K_0(\qgr(A))$ is isomorphic as an ordered abelian group to $\bZ[\theta]$ via the map $[\pi^* M] \mapsto q_M(\theta)$.

Furthermore, under the isomorphism(s),  the functor $\cM \mapsto \cM(1)$ on $\qgr(A)$ corresponds to multiplication by 
$t^{-1}$ and multiplication by $\theta^{-1}$. 
\end{thm}

\subsubsection{}
 \label{sssect.dim=2} 
Suppose $g=2$. Then  $\gr(A)$ is equivalent to $\gr(k[x_0,x_1])$ where $k[x_0,x_1]$ 
is a weighted polynomial ring  with $\deg(x_0)=1$ and  $\deg(x_1)=m$ so $f(t)=(1-t)(1-t^m)$. 
As remarked in  \S\ref{ssect.dim=2}, this case is well understood by \cite{ps} and not so interesting. 

From now on we assume that  $g\ge 3$.

\subsubsection{}
 \label{sect.roots}
Descartes' rule of signs implies that $f(t)$ has either 0 or 2 positive real roots. 
The hypothesis that $g \ge 3$ implies $f(1)<0$. Since $f(0)>0$, we conclude that $f(t)$ has two positive roots, 
$\theta^{-1}>1$ and $\theta \in (0,1)$, say.

\subsubsection{}
In general, $f(t)$ need not be irreducible. For example, if $A=k\langle x,y,z\rangle/(xz+zx+y^2)$ with $\deg(x,y,z)=(5,6,7)$,
then
$$
f(t)=t^{12}-t^7-t^6-t^5+1=(t^2-t+1)(t^{10}+t^9-t^7-t^6-t^5-t^4-t^3+t+1).
$$
There are reducible $f(t)$ of smaller degree but this particular example is interesting because the degree-10 factor is Lehmer's
polynomial \cite{EH}. Lehmer's number is $\theta^{-1} \approx 1.17628$.

\subsection{}
The proof of Theorem \ref{thm.main} uses the next result  with $d_i=\deg(x_i)$. 

\begin{prop}\label{minimal}
Let $d_1,\ldots,d_g$  be non-negative integers. Suppose that $g \ge 3$, that $d_i+d_{g+1-i}=d$ for all $i$, 
and ${\sf gcd}\{d_1,\ldots,d_{g}\}=1$.  
The polynomial
\begin{equation}
\label{defn.f2}
f(t):= t^d - \sum_{i = 1}^g t^{d_i} + 1
\end{equation}
has a unique root of maximal modulus and that root is real.
\end{prop}
 
We give two proofs of this result. That in \S\ref{proof1} uses a directed graph associated to $d_1,\ldots,d_g$
and the Perron-Frobenius theorem. That in  \S\ref{proof2} uses 
elementary trigonometric identities, but only works when $g \ge 4$.

\section{Proof of Theorem \ref{thm.main} modulo Proposition \ref{minimal}}

\begin{lem}\label{converge}
Let $\sum_{n=0}^\infty c_n z^n$ be a  power series in which $c_n>0$ for all $n \gg 0$. Suppose 
\begin{enumerate}
  \item 
 $\sum_{n=0}^\infty c_n z^n$ has radius of convergence $R>0$ and on the disk $|z|<R$ it converges to a rational function 
  $s(z)$ that has a simple pole at $z=R$;
  \item 
  all other poles of $\,\sum_{n=0}^\infty c_n z^n$ have modulus $>R$.
\end{enumerate}
Then
$$
\lim_{n\to \infty} \frac{c_n}{c_{n+1}} =R.
$$
\end{lem}
\begin{proof}
There are polynomials $p(z)$ and $q(z)$, neither divisible by $R-z$, such that
$$
s(z)=\frac{p(z)}{(R-z)q(z)} =\frac{\alpha}{R-z} +\frac{r(z)}{q(z)}
$$
where $\alpha \in \bC^\times$, $r(z)$ is a polynomial, and $r(z)/q(z)$ has a Taylor series expansion $\sum_{n=0}^\infty  b_n z^n$ with radius of convergence $>R$ by (2). Since $\sum_{n=0}^\infty b_nR^n$ converges, $\lim_{n \to \infty}b_nR^n =0$.

Since 
$$
s(z)=\frac{\alpha}{R}\sum_{n=0}^\infty \frac{z^n}{R^n} +\sum_{n=0}^\infty b_n z^n
$$
for $|z|<R$, 
$$
c_n=\frac{\alpha}{R^{n+1}} + b_n.
$$
Therefore
$$
\lim_{n\to \infty} \Bigg( \frac{c_n}{c_{n+1}} \Bigg) \; = \; \lim_{n\to \infty} \Bigg(  \frac{\alpha R +b_nR^{n+2}}{\alpha +b_{n+1}R^{n+2}} \Bigg) 
\; = \; \frac{\alpha R}{\alpha} \; = \;R,
$$
as claimed.
\end{proof}

\subsection{Notation}

Let $A$ be the algebra in (\ref{defn.A}). Let $d_i:=\deg (x_i)$ and $d: = \deg(b)$. 
 We assume that  $g \ge 3$, that ${\sf gcd}\{d_1, \ldots, d_g\} = 1$, and $d=d_i + d_{g + 1 - i}$ for all $i$. 
 
We suppose $d_1 \leq d_2 \leq \cdots \leq d_g$ for simplicity. We write $n_i$ for the number of $x_j$'s whose degree is $i$. Thus  
$$
f(t) = t^d - \sum_{i = 1}^{d-1} n_i t^i + 1.
$$

Because $A$ is a domain  \cite[Thm. 0.2]{zh} and 1 is the greatest common divisor of the degrees of its generators, $A_n \ne 0$ for all  $n \gg 0$.

We write  $a_n := \dim_k(A_n)$.

\begin{lem}\label{converge2}
For all $m \geq 1$,
$$
\lim_{n \to \infty} \frac{a_n}{a_{n+m}} = \theta^m.
$$
\end{lem}
\begin{proof}
 Since 
$$
\frac{a_n}{a_{n+m}} = \frac{a_n}{a_{n+1}}\frac{a_{n+1}}{a_{n+2}} \cdots \frac{a_{n + m -1}}{a_{n+m}}
$$
for $n \gg 0$, it suffices to prove the result for $m = 1$.
Since $H_A(t) = \sum_{i = 0}^\infty a_i t^i$
satisfies the conditions of Lemma \ref{converge} for $R = \theta$, the result follows from the conclusion of Lemma \ref{converge}.
\end{proof}

\begin{prop}\label{positivity}
If $M \in \gr(A)$, then $q_M(\theta) \geq 0$.
\end{prop}
\begin{proof}
Write $q_M(t) = \sum_{i = -s}^s p_i t^i$ and define $e_i := \sum_{j = -s}^s p_j a_{i-j}$.
Then 
$$
H_M(t) \; = \; q_M(t)H_A(t)= \left( \sum_{i = -s}^s p_i t^i \right) \left( \sum_{i = 0}^\infty a_i t^i \right) \; = \; \sum_{i = -s}^\infty e_i t^i. 
$$
As $m \to \infty$,  
$$
\frac{e_m}{a_m} \; = \;  \sum_{j = -s}^s \left(\frac{a_{m-j}}{a_m}\right)p_j \;  \longrightarrow  \;
\sum_{j = -s}^s p_j \theta^j  \; = \; q_M(\theta).
$$
 Since  $e_i=\dim(M_i)$, $\set{e_m / a_m}_{m \gg 0}$ is a sequence of non-negative numbers its limit, $q_M(\theta)$, is $\geq 0$.
\end{proof}

\begin{lem}
\label{fdiml}
Let $M \in \gr(A)$. The following are equivalent:
\begin{enumerate}
\item $M \in \fdim(A)$;
\item $f(t)$ divides $q_M(t)$;
\item $q_M(\theta) = 0$.
\end{enumerate}
\end{lem}
\begin{proof}
$(1) \Rightarrow (2)$ If $\dim_k(M)<\infty$, then $H_M(t)\in \bN[t,t^{-1}]$ so $q_M(t)$ is a multiple of $f(t)$.

$(2) \Rightarrow (3)$ If $f(t)$ divides $q_M(t)$ then $q_M(\theta) = 0$ since $\theta$ is a root of $f(t)$.

$(3) \Rightarrow (1)$ Suppose $q_M(\theta) = 0$ but $\dim_k(M)=\infty$. The power series $H_M(t)$ has non-negative coefficients and a finite radius of convergence $R \leq 1$. Since $H_M(t)= q_M(t)H_A(t)$, $q_M(\theta) = 0$ and $\theta$ is a simple pole of $H_A(t)$ and the only pole of $H_A(t)$ in the interval $[0,1]$, $H_M(t)$ has no poles in the interval $[0,1]$. 
This contradicts Pringsheim's Theorem \cite[Theorem IV.6]{ancomb} which says that $H_M(t)$ has a pole at $t = R$. 
We therefore conclude that $\dim_k(M)<\infty$.
\end{proof}

\subsection{} 
The converse of Proposition \ref{positivity} is not true: there are Laurent polynomials $p \in \bZ[t, t^{-1}]$ such that $p(\theta) \geq 0$ but $p(t)$ is not equal to $q_M(t)$ for any $M \in \gr(A)$. The following example illustrates this fact.

\begin{ex}
Let $A = k\la x, y, z\ra/(xz + y^2 + zx)$ with $\deg(x) = \deg(y) = \deg(z) = 1$. In this case, $f(t) = 1 - 3t + t^2$ and $\theta = \frac{1}{2}(3 - \sqrt{5})$.

Let $p(t) = -3 + 13t - 4t^2$. Then $p(\theta) = 1 + \theta > 0$; but
$$
p(t)H_A(t) = -3 + 4t^2 + 11t^2 + \cdots
$$
has a negative coefficient so  is not the Hilbert series of any module.

However, if $M = A \oplus A(-1)$, then $q_M(t) = 1 + t$ and $q_M(t) - p(t) = 4 - 12t + 4t^2 = 4f(t)$, so
$$
\big(q_M(t) - p(t)\big)H_A(t) = 4.
$$
In other words, $H_M(t)$ and $p(t)H_A(t)$ differ only in the the first term. The following lemma is a generalization of this fact.
\end{ex}

\begin{lem}
\label{lem.existsM}
Let $p \in \bZ[t,t^{-1}]$. If $p(\theta) > 0$, then there is an $M$ in  $\gr(A)$ such that $q_M(t) - p(t) \in (f)$.
\end{lem}
\begin{proof}
It suffices to show that $t^sq_M(t)-t^sp(t) \in (f)$ for some integer $s$. 
Since $q_{M(-s)}(t)=t^sq_M(t)$, we can, and will, assume $p(t) \in \bZ[t]$. 

Write $p(t) = \sum_{i = 0}^s p_i t^i$. Define integers $b_i$, $i \ge 0$, by the requirement that
\begin{equation}
\label{defn.bi}
\sum_{i=0}^\infty b_it^i : = p(t) H_A(t).
\end{equation}
Therefore
$$
p(t)=f(t) \sum_{i=0}^\infty b_it^i  = \Bigg( \sum_{i=0}^\infty b_it^i\Bigg) \Bigg(1-\sum_{j=1}^{d-1} n_j t^j +t^d\Bigg)  .
$$
Equating coefficients gives 
\begin{equation}
\label{eq.coeffs}
p_i=b_i+b_{i-d}-\sum_{j=1}^{d-1}n_j b_{i-j}
\end{equation}
for all $i \ge 0$ with the convention that $p_i=0$ for $i>s$ and $b_i=0$ for $i<0$. 

Since $a_j \ne 0$ for  $j \gg 0$,  
$$
\lim_{j\to \infty} \bigg( \frac{b_j}{a_j} \bigg) \; = \;  \lim_{j\to \infty} \Bigg(  \sum_{i=0}^s \left(\frac{a_{j-i}}{a_j}\right)p_i \Bigg) \; =\; \sum_{i = 0}^sp_i\theta^i \; = \; p(\theta) \; > \; 0.
$$
Therefore
$$
\lim_{j \to \infty} \bigg(  \frac{b_j}{b_{j+1}} \bigg) \;=\;  \lim_{j \to \infty} \Bigg(  \frac{b_j}{a_j}  \frac{a_{j+1}}{b_{j+1}} \frac{a_j}{a_{j+1}}\Bigg) \; = \;  p(\theta) p(\theta)^{-1} \theta  \; = \;  \theta.
$$

It follows that there is an integer $m \ge s$ such that 
$$
0<b_{m-d}<b_{m+1-d}< \cdots < b_i < b_{i+1}< \cdots.
$$
 We fix such an $m$. 

Define the integers $r_i$ by
$$
\sum_i r_it^i := p(t) - \left(\sum_{i = 0}^m b_it^i\right)f(t). 
$$
We have
\begin{align*}
\sum r_it^i & \; = \;   p(t) - \left(\sum_{i = 0}^m b_it^i\right)\left(1 - \sum_{j = 1}^{d-1} n_j t^j + t^d\right)
\\
& \; = \;  
p(t) - \sum_{i = 0}^m \left[b_i - \sum_{j = 1}^{d-1} n_j b_{i-j}+ b_{i-d}\right]t^i  + \sum_{i = m+1}^{m+d}\left[\sum_{j = i-m}^{d-1} 
n_j b_{i-j} - b_{i-d}\right]t^i
\\
& \; = \;  
 \sum_{i = m+1}^{m+d}\left[\sum_{j = i-m}^{d-1} n_j b_{i-j} - b_{i-d}\right]t^i \qquad \hbox{by (\ref{eq.coeffs})}.
\end{align*}
Thus,  
\begin{equation}
\label{defn.ri}
r_i = \sum_{j = i - m}^{d-1} n_j b_{i-j} - b_{i-d}
\end{equation}
for $m+1 \le i \le m+d$ and $r_i$ is 0 if $i \notin [m+1,m+d]$. 

Suppose $d=2$, i.e., all generators have degree 1. Then  
\begin{align*}
\sum r_i t^i & \; = \; r_{m+1}t^{m+1}+r_{m+2}t^{m+2} 
\\
&\; = \;  (n_1b_m-b_{m-1})t^{m+1}+(-b_m)t^{m+2}
\\
&\; = \;  ((n_1-1)b_m-b_{m-1})t^{m+1}+b_m(1-t)t^{m+1}
\end{align*}
which equals $q_M(t)$ if 
$$
M= A(-m-1)^{(n_1-1)b_m-b_{m-1}} \oplus \Bigg(\frac{A}{x_1A}\Bigg)(-m-1)^{ b_m}.
$$
Note that there really is such a module because $(n_1-1)b_m \ge b_m > b_{m-1}>0$.

From now on we assume that $d \ge 3$. Thus $d_1\ne d_g$ and $d_1+d_g=d$.

We will prove the result by showing there are modules $L,N \in \gr(A)$ such that
\begin{equation}
\label{qN}
 \sum_{i=m+d_1+1}^{m+d_g} r_i t^i \;=\; q_N(t)
\end{equation}
 and 
\begin{equation}
\label{qL}
 \sum_{i = m+1}^{m+d_1}r_it^i  \; + \; \sum_{i = m+d_g+1}^{m+d}r_it^i \; = \; q_L(t).
\end{equation}

First we prove (\ref{qN})  by showing that $r_i \ge 0$ when $m+d_1+1 \le i \le m+d_g$. For such an $i$,
$i-m \le d_g$ so the term $n_{d_g}b_{i-d_g}$ appears in the expression $r_i=\sum_{j=i-m}n_jb_{i-j}-b_{i-d}$;
that is sufficient to imply that $r_i > 0$ because $n_{d_g}b_{i-d_g} -b_{i-d} \ge b_{i-d_g}-b_{i-d} >0$.  Thus, there
is a module $N \in \gr(A)$ such that (\ref{qN}) holds. 

Now we prove the existence of $L$.   
Because $d_1+d_g=d$, 
$$
 \hbox{(\ref{qL})} \; = \; \sum_{i = m+1}^{m+d_1}(r_i+r_{i+d_g}t^{d_g})t^i. 
$$
However, $n_j=0$ for all $j \ge d_g+1$ so, when $m + 1 \leq i \leq m+d_1$,
$$
r_i + r_{i + d_g}t^{d_g} \; = \;  r_i - b_{i+d_g-d}t^{d_g} \; = \;  r_i - b_{i - d_1} + b_{i-d_1}(1-t^{d_g}),
$$
Therefore 
\begin{align*}
 \hbox{(\ref{qL})} &  \; = \;  \sum_{i = m+1}^{m+d_1}(r_i - b_{i - d_1})t^i +  \sum_{i = m+1}^{m+d_1}b_{i-d_1}(1-t^{d_g})t^i
 \\
 &  \; = \; q_L(t)
 \end{align*}
where 
$$
L \; = \; \Bigg(\bigoplus_{i = m+1}^{m+d_1}A(-i) ^{r_i - b_{i-d_1}}\Bigg)
 \; \oplus \; 
 \Bigg( \bigoplus_{i = m+1}^{m+d_1}\bigg( \frac{A}{x_g A}\bigg)(-i)^{b_{i-d_1}} \Bigg) 
$$
{\it provided} $r_i-b_{i-d_1}$ and $b_{i-d_1}$ are non-negative. 

Certainly, $b_{i-d_1} >0$ because $i-d_1 \ge m+1-d$.

If $m+1 \leq i  \leq m+d_1$, then
$$
r_i  \; \geq  \; n_{d_1}b_{i-d_1} + n_{d_g}b_{i - d_g} - b_{i-d} \; \geq \; b_{i-d_1}
$$
so $r_i-b_{i-d_1} \ge 0$. 
\end{proof}

\subsection{The proof of Theorem \ref{thm.main} modulo Proposition \ref{minimal}}

The proof is similar to that in \cite[\S2]{mosmi} so we refer the reader there for more detail.

By definition, $\qgr(A) =\gr(A)/\fdim(A)$, so the localization sequence for $K$-theory gives an exact sequence
\begin{equation}
\label{loc.seq}
K_0(\fdim(A))  \stackrel{\iota}{\longrightarrow}   K_0(\gr(A)) \longrightarrow K_0(\qgr(A))   \longrightarrow 0.
\end{equation}
Because every module in $\gr(A)$ has a finite resolution by finitely generated free (graded) modules, $K_0(\qgr(A))  
\cong \ZZ[t,t^{-1}]$. Because every module in $\fdim(A)$ is finite dimensional it has a composition series and its composition
factors are shifts of the trivial module $A/A_{\ge 1}$. Thus, by d\'evissage, $K_0(\fdim(A)) \cong K_0(\gr(k))\cong \ZZ[t,t^{-1}]$.
Because the inclusion $\fdim(A) \to \gr(A)$ commutes with the degree-shift functor, the map 
$K_0(\fdim(A))  \stackrel{\iota}{\longrightarrow}   K_0(\gr(A))$ is a homomorphism  of $\ZZ[t,t^{-1}]$-modules. Under the isomorphism
$K_0(\gr(k))\cong \ZZ[t,t^{-1}]$, $[k] \mapsto 1$. Therefore $K_0(\qgr(A))$ is isomorphic to $\ZZ[t,t^{-1}]$ modulo the ideal generated by 
$[A/A_{\ge 0}]$. It follows from the minimal resolution of $A/A_{\ge 0}$, that $[A/A_{\ge 0}]=f(t)$. 
Therefore the map
\begin{equation}
\label{isom.MS}
K_0(\qgr(A)) \to \frac{\bZ[t, t^{-1}]}{(f)}, \qquad [\pi^* M] \mapsto \ol{q_M(t)},
\end{equation}
is an isomorphism of abelian groups and $[\cM(1)]=t^{-1}[\cM]$ under this isomorphism.

Under the isomorphism (\ref{isom.MS}), the positive cone in $K_0(\qgr(A))$ is mapped to 
$\{\ol{q_M(t)} \; | \; M \in \gr(A)\}$. To show that (\ref{isom.MS}) is an isomorphism of 
{\it ordered} abelian groups we must show that
\begin{equation}
\label{same.pos.cone}
\{\ol{p}\; | \; p(\theta) > 0\} \cup \set{0} = \{ \ol{q_M(t)} \; | \; M \in \gr(A)\}.
\end{equation}

Let $M \in \gr(A)$. By Proposition \ref{positivity}, $q_M(\theta)  \geq 0$. If $q_M(\theta)>0$, then $\ol{q_M(t)}$ is in the left-hand side of (\ref{same.pos.cone}). If $q_M(\theta) = 0$, then $f(t)$ divides $q_M(t)$ by Lemma \ref{fdiml},  whence $\ol{q_M(t)}=0$. Thus, the right-hand side of (\ref{same.pos.cone})
is contained in  the left-hand side of (\ref{same.pos.cone}). 

If $p \in \bZ[t^{\pm 1}]$ and $p(\theta) > 0$, then $\ol{p} = \ol{q_M}$ for some $M \in \gr(A)$ by Lemma \ref{lem.existsM} so $\ol{p}$ is in
 the right-hand side of (\ref{same.pos.cone}). It is clear that $0$ is in  the right-hand side of (\ref{same.pos.cone}). Thus, the left-hand side of (\ref{same.pos.cone}) is contained in  the right-hand side of (\ref{same.pos.cone}). Hence (\ref{isom.MS}) is an isomorphism of ordered abelian groups.

Suppose $f$ is irreducible. The composition
\begin{equation}
\label{same.pos.cone2}
K_0(\qgr(A))  \to \frac{\bZ[t, t^{-1}]}{(f)} \to   \bZ[\theta], \quad [\pi^*M] \mapsto q_M(\theta),
\end{equation}
is certainly an isomorphism of abelian groups. By (\ref{same.pos.cone}), the image of the positive cone in $K_0(\qgr(A))$ under
this composition is $\bR_{\geq 0} \cap \bZ[\theta]$, the positive cone in $\bZ[\theta]$. Hence (\ref{same.pos.cone}) is an isomorphism  of ordered abelian groups and $[\cM(1)] = \theta^{-1}[\cM]$ under the isomorphism.
\hfill $\square$

\section{First proof of Proposition \ref{minimal}}\label{proof1}

\subsection{}
\label{ssect.strategy}
We recall the statement of Proposition \ref{minimal} and explain our strategy to prove it.

Let $d_1,\ldots,d_g$  be non-negative integers. Suppose $g \ge 3$, that $d_i+d_{g+1-i}=d$ for all $i$, 
and ${\sf gcd}\{d_1,\ldots,d_{g}\}=1$. Let $n_j$ be the number of $d_i$ that equal $j$. 
We will show that the polynomial
\begin{equation}
\label{defn.f3}
f(t)=t^d-n_{d-1}t^{d-1}-\cdots -n_1t+1
\end{equation}
has a unique root of maximal modulus, and that root is real.

We will associate to the data $d_1,\ldots,d_{g}$ a particular finite directed graph $G$. An {\sf adjacency matrix} for $G$ is a square matrix whose rows and columns are labelled by the vertices of $G$ and whose $uv$-entry is the number of arrows from $v$ to $u$. The {\sf characteristic polynomial} of $G$ is 
$$
p_G(t):=\det(tI-M)
$$
where $M$ is an adjacency matrix for $G$. We will show that $p_G(t)=t^{\ell-d}f(t)$ where $\ell=d_1+\cdots+d_g$. 
We also show that $M$ is primitive, i.e., all entries of $M^n$ are positive for $n \gg 0$.  
We then apply the Perron-Frobenius theorem which says that a primitive matrix has a positive real
eigenvalue of multiplicity 1, $\rho$ say, with the property that $|\lambda|<\rho$ for all other eigenvalues $\lambda$. But the non-zero eigenvalues of $M$ are the roots of $f(t)$. Since we already know that $f(t)$ has only two positive real 
roots,  $\theta <1$ and $\theta^{-1}>1$,  $\rho=\theta^{-1}$.  
Since the coefficient of $t^i$ in $f(t)$ is the same as that of $t^{d-i}$,  $f(t)=t^df(t^{-1})$. Thus $f(\lambda)=0$
if and only if $f(\lambda^{-1})=0$. Hence $\theta^{-1}$ is the unique root of $f(t)$ having largest modulus.

\subsection{}
We will use Theorem \ref{cycles} to compute the characteristic polynomial of  $G$. 
First we need some notation.  

A {\sf simple cycle} in $G$ is a directed path that begins and ends at the same vertex and does not pass through any vertex 
more than once. We introduce the notation for an arbitrary directed graph $G$:
\begin{enumerate}
\item
$v(G):=$ the number of vertices in $G$;
  \item{}
  $c(G):=$ the number of connected components in $G$;
  \item 
  $Z(G):= \{\hbox{simple cycles in $G$}\}$;
  \item
  $\Zol(G):= \{\hbox{subgraphs of $G$ that are a disjoint union of simple cycles}\}$.
\end{enumerate}

\begin{thm}
\label{cycles}
\cite[Theorem 1.2]{graph} 
Let $G$ be a directed graph with $\l$ vertices. Then 
\begin{equation}
\label{pG}
p_G(t) = t^\l + c_1t^{\l-1} + \cdots + c_{\l-1}t + c_\l
\end{equation}
where 
\begin{equation}
\label{ci}
c_i := \sum_{\begin{subarray} {c} 
 			Q \in \Zol(G)  \\
			v(Q)=i
      	          \end{subarray}}   			
	(-1)^{c(Q)}.
\end{equation}
\end{thm}

\subsubsection{Remark}
\label{rem.for.referee}
In Theorem \ref{cycles}, we note for future reference that 
$$
p_G(0)=c_\l =   \sum_{\begin{subarray} {c} 
 			Q \in \Zol(G)  \\
			v(Q)=v(G)
      	          \end{subarray}}   			
	(-1)^{c(Q)}.
$$	

\subsection{}
The $x_i$s are labelled so that $\deg(x_1) \le \cdots \le \deg(x_g)$.

The free algebra on $k \langle x_1,\ldots,x_g\rangle$ is the path algebra of the quiver with one vertex $\star$ and $g$ loops from $\star$ to $\star$ labelled $x_1,\ldots,x_g$. We replace each  loop $x_i$ by $d_i':=\deg(x_i)-1=d_i-1$ 
vertices labelled $x_{i1},\ldots,x_{id_i'}$ and arrows
$$
\xymatrix@C=40pt{
\star \ar[r]^{\alpha_{i0}} & x_{i1} \ar[r]^{\alpha_{i1}} & \cdots & \cdots \ar[r] & x_{id_i'} \ar[r]^{\alpha_{id_i'}} &\star
}
$$
 The graph obtained by this procedure is the graph associated to $k\la x_1, \ldots, x_g\ra$ in \cite{holdsis}.

\subsubsection{Example}
\label{eg.1}
If $A$ is generated by $x_1,x_2,x_3$ and $\deg(x_i)= i$, the associated graph is
$$
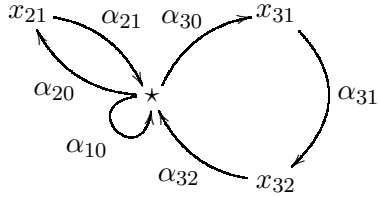
\xymatrix@R=20pt@C=30pt{
x_{21} \ar@/^10pt/@<-0.5ex>[dr]^{\alpha_{21}} & & x_{31} \ar@/^20pt/[dd]^{\alpha_{31}} \\
& \star \ar@/^10pt/[ul]^{\alpha_{20}} \ar@/^10pt/@<-0.5ex>[ur]^{\alpha_{30}} \ar@(l,d)[]_{\alpha_{10}} \\
& & x_{32} \ar@/^10pt/[ul]^{\alpha_{32}}
}
$$
 
\subsection{The second graph associated to $A$}
We now form a second directed graph, the vertices of which are the arrows in the previous graph. In the second graph there is an arrow from vertex $u$ to vertex $v$ if in the first graph the arrow $u$ can be followed by the arrow $v$, {\bf except} 
we do not include an arrow $\alpha_{1d_1'} \to \alpha_{g0}$.

We write $G$, or $G(A)$, for the second graph associated to $A$.

The second graph associated to Example \ref{eg.1} is  
$$
\xymatrix@=30pt{
& \alpha_{20} \ar@/^10pt/[r] & \alpha_{21} \ar[l] \ar[dr] \ar[dll]& \\
\alpha_{10} \ar@(ul,dl)[] \ar[ur] & & & \alpha_{30} \ar[dl]\\
& \alpha_{32} \ar[ul] \ar[uu] \ar[rru]& \alpha_{31} \ar[l]&
}
$$
Note the absence of an arrow from $\alpha_{10}$ to $\alpha_{30}$. 

\begin{prop}
\label{prop.irred}
If $u$ and $v$ are vertices in $G$, there is a directed path starting at $u$ and ending at $v$.
\end{prop}
\begin{proof}
There is a directed path $\alpha_{i0}\to \alpha_{i1} \to \cdots \to \alpha_{id_i'} \to \alpha_{i0}$ so the result is true if 
$u=\alpha_{ij}$ and $v=\alpha_{ik}$. There are also arrows 
$$
\alpha_{1d_1'} \to \alpha_{20}, \quad \alpha_{2d_2'} \to \alpha_{30}, \quad \ldots \quad \alpha_{g-1,d'_{g-1}} \to \alpha_{g0},
\quad
\alpha_{gd_g'} \to \alpha_{10}
$$
so the result is true if $u=\alpha_{i_1j_1}$ and $v=\alpha_{i_2j_2}$. 
\end{proof}

\begin{prop}
\label{prop.primitive}
Let $M$ be an adjacency matrix for $G$. Then every entry in $M^n$ is non-zero for $n \gg 0$.
\end{prop}
\begin{proof}
In the language of  \cite[Defn. 4.2.2]{LM}, Proposition \ref{prop.irred} says that $M$ is irreducible. 

The {\sf period} of a vertex $v$ in $G$ is the greatest common divisor of the non-trivial 
directed paths that begin and end at $v$. The {\sf period} of $G$ is the 
greatest common divisor of the periods of its vertices. Since there is a directed path of length $d_i=\deg(x_i)$ from $\alpha_{i0}$
to itself, the period of $G$ divides $\gcd\{d_1,\ldots,d_g\}$ which is 1. The period of $G$ is therefore 1. 
Thus, in the language of \cite[Defn. 4.5.2]{LM}, $M$ is aperiodic and therefore primitive \cite[Defn. 4.5.7]{LM}.
Hence \cite[Thm.  4.5.8]{LM} applies to $M$, and gives the result claimed.
\end{proof}

The Perron-Frobenius theorem \cite[Thm. 1, p.64]{G}  therefore applies to $M$ giving the following result.

\begin{cor}
\label{cor.PF}
The characteristic polynomial for $G$ has a unique eigenvalue of maximal modulus and that eigenvalue is simple and real.
\end{cor}

Our next goal, achieved in Proposition \ref{charpoly}, is to show that $p_G(t)=t^{\l-d}f(t)$ for a suitable $\l$.

\subsection{Other graphs associated to $A$}
We now write $\cX:=\{x_1, \ldots, x_g\}$ and define the directed graph $\widehat{\cX}$ by declaring that its vertex set 
is $\cX$ and there is an arrow $x_i \to x_j$ for all $(x_i,x_j) \in \cX^2-\{(x_1,x_g)\}$. 
For each non-empty subset  $X \subset \cX$ let $\widehat{X}$  be the full subgraph of $\widehat{\cX}$
with vertex set $X$.

\subsubsection{}
If $g = 4$, then
$$ 
 \{x_1, x_2, x_3\}\,\widehat{\phantom{x}} \quad = \quad 
 
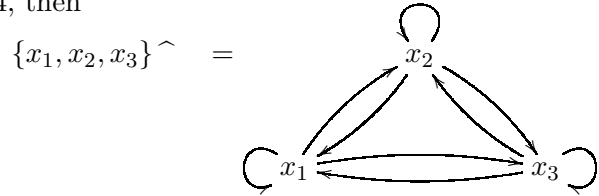
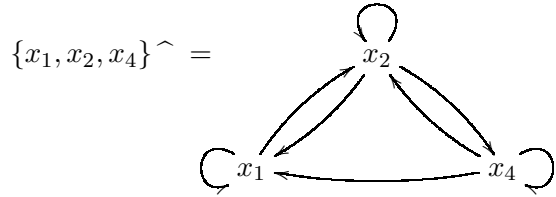
\xymatrix@R=30pt@C=30pt{
& x_2 \ar@(ur,ul)[] \ar@/^5pt/[dr] \ar@/^5pt/[dl]& \\
x_1 \ar@/^5pt/[ur] \ar@/^5pt/[rr] \ar@(ul,dl)[] & & x_3 \ar@/^5pt/[ul] \ar@/^5pt/[ll] \ar@(ur,dr)[]}
$$
and 
$$
\{x_1, x_2, x_4\}\,\widehat{\phantom{x}} \; = \; 
\xymatrix@R=30pt@C=30pt{
& x_2 \ar@(ur,ul)[] \ar@/^5pt/[dr] \ar@/^5pt/[dl]& \\
x_1 \ar@/^5pt/[ur]  \ar@(ul,dl)[] & & x_4 \ar@/^5pt/[ul] \ar@/^5pt/[ll] \ar@(ur,dr)[]}
$$

\vskip .2in

\begin{lem}
\label{lem.char.poly}
Let $X \subset \{x_1,\ldots,x_g\}$.
The constant term in the characteristic polynomial for $\widehat{X}$ is 
$$
p_{\widehat{X}}(0) \; = \; \begin{cases}
1 & \text{if $X=\{x_1,x_g\}$}
\\
-1 & \text{if $|X|=1$}
\\
0 & \text{otherwise.}
\end{cases}
$$
\end{lem}
\begin{proof}
Let $M$ be an adjacency matrix of  $\widehat{X}$. Then the constant term in the characteristic polynomial for $\widehat{X}$ is  $p_{\widehat{X}}(0)=(-1)^{|X|}\det(M)$. 

If $|X|=1$, then $\widehat{X}$ consists of one vertex with a single loop so $M=(1)$ whence $p_{\widehat{X}}(0)=-1$.

If $X=\{x_1,x_g\}$, then $\widehat{X}$ has vertices $x_1$ and $x_g$, an arrow from $x_g$ to $x_1$, and a loop at each vertex.
Hence $\binom{ 1 \; 1}{0 \; 1}$ is an adjacency matrix for $\widehat{X}$ and the constant term is 1.

If $|X|=2$ and $X\ne \{x_1,x_g\}$, then the adjacency matrix for $\widehat{X}$ is $\binom{ 1 \; 1}{1 \; 1}$  so the constant term is 0.

Suppose $|X| \ge 3$.
If  $\set{x_1, x_g} \subseteq X$, then $M$ has a single off-diagonal 0 and all its other entries are 1; in particular, $M$ is singular so 
the constant term is 0. If $\set{x_1, x_g} \not\subseteq X$, then every entry in $M$ is 1 so $M$ is singular 
and the constant term is 0.
\end{proof}

\subsection{The paths $\beta_1,\ldots, \beta_g$ in $G$}
For each $1 \leq i \leq g$, let $\beta_i$ be the path 
$$
\alpha_{i0} \to \alpha_{i1} \to \cdots \to \alpha_{id_i'}.
$$
In example \ref{eg.1}, $\beta_1$ is the trivial path at vertex 
$\alpha_{10}$, $\beta_2$ is the arrow 
$\alpha_{20} \longrightarrow \alpha_{21}$, 
and $\beta_3$ is the path  $\alpha_{30} \longrightarrow \alpha_{31} \longrightarrow \alpha_{32}$.

\begin{prop}
Let $i_1,\ldots,i_m$ be pairwise distinct elements of $\{1, \ldots, g\}$ such that 
$(1,g) \not\in \{(i_m,i_1), (i_1,i_2), \ldots, (i_{m-1},i_m)\}$. 
Then there is a simple cycle in $G$ of the form 
\begin{equation}\label{cycleQ'1}
\beta_{i_1} \to \beta_{i_2} \to \cdots \to \beta_{i_m} \to \alpha_{i_10}
\end{equation}
and every simple cycle in $G$ is of this form, up to choice of starting point.
\end{prop}
\begin{proof}
Let $r,s \in \{1,\ldots,g\}$ and assume $r \ne s$. If $(r,s) \ne (1,g)$, then there is an arrow from $\alpha_{rd_r'}$, 
the vertex at which $\beta_r$ ends, to $\alpha_{s0}$, the vertex at which $\beta_s$ starts; hence there is a path ``traverse $\beta_r$ then traverse $\beta_s$''; we denote this path by $\beta_r \to \beta_s$. It follows that there is a path of the form  (\ref{cycleQ'1}).

Let $p$ be a simple cycle in $G$. A simple cycle passes through a vertex $\alpha_{ij}$ if and only if it passes through $\alpha_{i0}$. Every simple cycle that passes
through $\alpha_{i0}$ contains $\beta_i$ as a subpath because there is a unique arrow starting at $\alpha_{ij}$ for all $j=0,\ldots,d_i'-1$.  Hence $p$ is of the form (\ref{cycleQ'1}).
\end{proof}

\begin{lem}
There is a bijection $\Phi: Z(\widehat{\cX}) \to Z(G)$ defined by
\begin{equation}\label{cycleQ'2}
\Phi(x_{i_1} \to   \cdots \to x_{i_m} \to x_{i_1})\; := \; \beta_{i_1} \to   \cdots \to \beta_{i_m} \to \alpha_{i_10}
\end{equation}
whose inverse is
\begin{equation}\label{cycleQ(G')1}
\Phi^{-1}(\beta_{i_1} \to   \cdots \to \beta_{i_m} \to \alpha_{i_10}) := x_{i_1} \to   \cdots \to x_{i_m} \to x_{i_1}.
\end{equation}
\end{lem}
\begin{proof}
We need only check that $\Phi$ and its purported inverse are  well-defined.
Because $\widehat{\cX}$ does not contain an arrow $x_1 \to x_g$ and $G$ does not contain an arrow 
$\alpha_{1d_1'} \to \alpha_{g0}$, the right-hand sides of (\ref{cycleQ'2}) and  (\ref{cycleQ(G')1})  {\sf are} simple cycles. \end{proof}

The next result is obvious.

\begin{prop}\label{bij}
The function $\Phi$ extends to a bijection $\Phiol:\Zol(\widehat{\cX}) \to \Zol(G)$ defined by
$$
\Phiol(E_1 \sqcup \cdots \sqcup E_m) \; := \; \Phi(E_1) \sqcup \cdots \sqcup \Phi(E_m)
$$
for disjoint simple cycles $E_1,\ldots,E_m$ in $\widehat{\cX}$. 
Furthermore, $c(E) = c(\Phiol(E))$ for all $E \in\Zol(\widehat{\cX})$.
\end{prop}

\subsection{}

The {\sf support} of a subgraph $Q$ of $G$ is 
$$
\supp(Q):=\{x_i \; |  \; \hbox{$\beta_i$ is a path in $Q$}\}.
$$
Notice that we are using the word ``support'' in a non-standard way here; for example, although $\beta_i$ passes through $\deg(x_i)$ vertices, 
$\supp(\beta_i)=\{x_i\}$, a singleton, and $v(\beta_i)=\deg(x_i)$.  
For each non-empty subset $X \subset \set{x_1, \cdots, x_g}$ let
$$
\Zol(G,X) := \{Q \in \Zol(G) \; | \; \supp(Q)=X\}
$$
and let $d(X) = \sum_{x \in X}\deg(x)$. Thus, $v(Q)=d(X)$ for all $Q \in \Zol(G,X)$.

\begin{prop}\label{charpoly}
Let $\l = \sum_{i=1}^g \deg(x_i)$. The characteristic polynomial of $G$ is $t^{\l-d}f(t)$.
\end{prop}
\begin{proof}		
The characteristic polynomial of $G$ is  
$$
p_{G}(t) = t^\l + c_1 t^{\l-1} + \cdots + c_{\l-1}t + c_\l
$$
where $\l = v(G)=\sum_{i=1}^g d_i$ and 
\begin{equation}
\label{orig.sum}
c_i \; = \; \sum_{\begin{subarray} {c} 
 			Q \in \Zol(G)  \\
			v(Q)=i
      	          \end{subarray}}   (-1)^{c(Q)}
\; = \; \sum_{\begin{subarray} {l} 
 					X \subset \cX \\
					d(X)=i
      				\end{subarray}} 
				\left(\sum_{Q \in \Zol(G,X)} (-1)^{c(Q)}\right).	          
\end{equation}	          
Since $\Zol(G,X) = \{\ol{\Phi}(E) \; | \; E \in \Zol(\widehat{X}) \; \hbox{and}  \; v(E)=|X|\}$ we have
\begin{equation}
\label{change.sum}
 \sum_{Q \in \Zol(G,X)} (-1)^{c(Q)} \; \; = \; \; \sum_{\begin{subarray} {c} 
 					E \in \Zol(\widehat{X})  \\
					v(E)=|X|
      				\end{subarray}}    (-1)^{c(E)}.
\end{equation}
Since $\widehat{X}$ is the full subgraph of $\widehat{\cX}$ with vertex set $X$, $v(\widehat{X})=|X|$.
It now follows from Remark \ref{rem.for.referee} that  the right-hand side of (\ref{change.sum}) is $p_{\widehat{X}}(0)$. 
Hence by Lemma \ref{lem.char.poly}, $$
c_i = 
\begin{cases}
1 & \text{if $i = d_1 + d_g = d$,} \\
-n_i & \text{if $1 \leq i \leq d_g$,} \\
0 & \text{otherwise}.
\end{cases}
$$
Thus $p_{G}(t) = t^\l - n_1 t^{\l-1} - \cdots - n_{d - 1}t^{\l-d+1} + t^{\l-d}  = t^{\l-d}f(t)$, as claimed.
\end{proof}

As explained  at the end of \S\ref{ssect.strategy}, 
Proposition \ref{minimal} follows from Proposition \ref{charpoly} and Corollary \ref{cor.PF}.

\subsubsection{Example}
\label{eg.234}
In order to clarify some of the technicalities in this section, we will compute the coefficient $c_5$ in
$p_G(t) = t^9 + c_1t^8 + \cdots + c_8t + c_9$ where $G$ is the  second graph associated to the algebra
$A=k\langle x_1,x_2,x_3\rangle/(b)$ where $\deg(x_i) = i+1$. First, $G$ is
$$
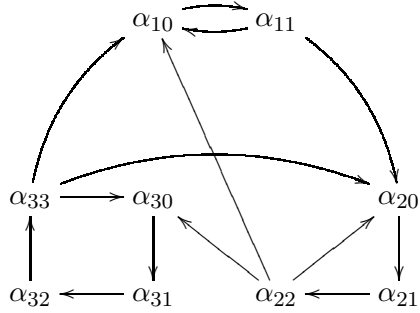
\xymatrix{
&\alpha_{10} \ar@/^5pt/[r]  & \alpha_{11} \ar@/^5pt/[l] \ar@/^10pt/[rdd] 
\\ 
\\
\alpha_{33} \ar[r] \ar@/^10pt/[uur] \ar@/^16pt/[rrr] & \alpha_{30} \ar[d]  && \alpha_{20}\ar[d] & \\
\alpha_{32}  \ar[u] & \alpha_{31} \ar[l]  &\alpha_{22} \ar[ur] \ar[ul] \ar[luuu] &\alpha_{21} \ar[l]  \\
}
$$
There are two subgraphs of $G$ that have exactly five vertices and are disjoint unions of simple cycles, namely
$$
Q_1 =
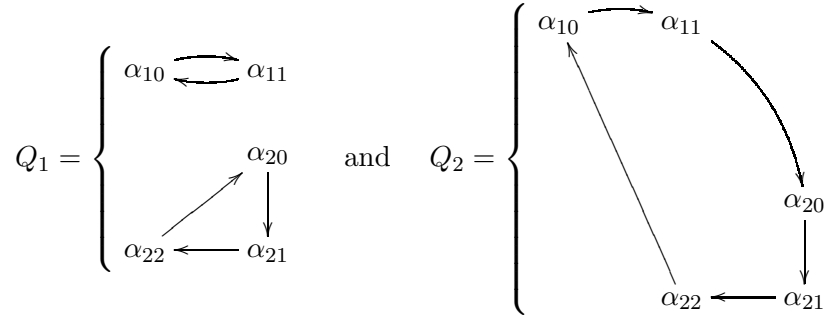
\begin{cases}
 \xymatrix{
\alpha_{10} \ar@/^5pt/[r]  & \alpha_{11} \ar@/^5pt/[l]
}
\\
\\ 
 \xymatrix{
& \alpha_{20} \ar[d] 
\\
 \alpha_{22} \ar[ur] & \alpha_{21}  \ar[l] 
}
\end{cases}
\text{ and } \quad
Q_2 =
\begin{cases}
 \xymatrix{
 \alpha_{10} \ar@/^5pt/[r]  & \alpha_{11}  \ar@/^10pt/[rdd] \\ 
\\
 & &  \alpha_{20}  \ar[d] \\
 & \alpha_{22}  \ar[luuu] &\alpha_{21} \ar[l] 
}
\end{cases}
$$
The only subset $X$ of $\cX = \set{x_1, x_2, x_3}$ such that $d(X) = 5$ is $X = \set{x_1,x_2}$. The graph $\widehat{\cX}$ is
$$
 
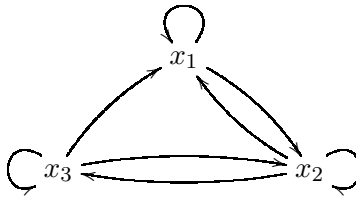
\xymatrix@R=30pt@C=30pt{
& x_1 \ar@(ur,ul)[] \ar@/^5pt/[dr]  & \\
x_3 \ar@/^5pt/[ur] \ar@/^5pt/[rr] \ar@(ul,dl)[] & & x_2 \ar@/^5pt/[ul] \ar@/^5pt/[ll] \ar@(ur,dr)[]}
$$

Since $Q_1 = \ol{\Phi}(E_1)$ and $Q_2 = \ol{\Phi}(E_2)$ where 

$$
E_1 =
\begin{cases}
 \xymatrix@R=30pt@C=30pt{
x_1 \ar@(ur,dr)[] & \\
x_2 \ar@(ur,dr)[]}
\end{cases}
\quad \text{ and } \quad
E_2 = 
\begin{cases}
\xymatrix@R=30pt@C=30pt{
x_1 \ar@/^5pt/[dr]  & \\
 & x_2 \ar@/^5pt/[ul] }
\end{cases}
$$
equations (\ref{orig.sum}) and (\ref{change.sum}) give
\begin{align*}
c_5 &= (-1)^{c(Q_1)} + (-1)^{c(Q_2)}\\
&= (-1)^{c(E_1)} + (-1)^{c(E_2)}\\
&= 1 - 1\\
&= 0.
\end{align*}

\section{Second proof of  Proposition \ref{minimal}}\label{proof2}
The following proof of Proposition \ref{minimal} is made under the additional assumption that $g \ge 4$.

  \begin{lem}
 \label{lem.highschool4}
Let $\eta,\phi,\psi \in \bR$. The following inequalities hold:
 \begin{enumerate}
  \item 
  $3+\cos(2\eta)-\cos(\eta+\psi)-\cos(\eta-\psi)-\cos(\eta+\phi)-\cos(\eta-\phi)  \; \ge \; 0$;
  \item 
 $ 3+\cos(2\eta)-\cos(\eta+\psi)- 2\cos(\eta)- \cos(\eta-\psi)  \; \ge \; 0$;
  \item 
 $3+\cos(2\eta)- 2\cos(\eta+\psi)-2\cos(\eta-\psi)  \; \ge \; 0$.
\end{enumerate}
In each case the inequality is an equality if  and only if 
\begin{enumerate}
  \item 
  $\cos(\eta)=\cos(\phi)=\cos(\psi)=1$ or $\cos(\eta)=\cos(\phi)=\cos(\psi)=-1$;
  \item 
  $\cos(\eta)=\cos(\psi)=1$;
  \item 
  $\cos(\eta)=\cos(\psi)=1$ or $\cos(\eta)=\cos(\psi)=-1$.
\end{enumerate}  
 \end{lem}
 \begin{proof}
 The expression in (2) is obtained from that in (1) by taking $\phi=0$. 
 The expression in (3) is obtained from that in (1) by taking $\psi=\phi$. 
 The statements about equality in (2) and (3) follow from that about equality in (1).
 It therefore suffices to prove (1).

 Let $\lambda = \frac{1}{2}(\cos(\phi)+\cos(\psi))$. Since $-1 \le \lambda \le 1$, 
 $$
1+ ( \cos(\eta)-\lambda)^2 - \lambda^2    \; \ge \; 0.
  $$
  Therefore
\begin{align*}
 0 & \; \le \; 2+ 2( \cos(\eta)-\lambda)^2 -2 \lambda^2 
 \\
 &\;  = \;  2+ 2\cos^2(\eta)-4\lambda\cos(\eta).
\\
  & \; = \; 2+2\cos^2(\eta) -2\cos(\phi)\cos(\eta) - 2\cos(\psi)\cos(\eta)   
  \\
   & \; = \;  3+\cos(2\eta)-\cos(\eta+\phi)-\cos(\eta-\phi) -\cos(\eta+\psi)-\cos(\eta-\psi) .
\end{align*}
This proves the  inequality  in (1). The  inequality in (1) is an equality if and only if
$\lambda^2=1$ and $\cos(\eta)=\lambda$, i.e., if and only if either $\cos(\eta)=\lambda=1$ or $\cos(\eta)=\lambda=-1$; i.e., if and only if 
$\cos(\eta)=\cos(\phi)=\cos(\psi)=1$ or $\cos(\eta)=\cos(\phi)=\cos(\psi)=-1$.
\end{proof}

As always, $f(t)$ is the polynomial defined in (\ref{defn.f3}) or, equivalently, the polynomial in (\ref{defn.f2}).

\begin{prop}\label{equalmodulus}
Suppose $g \geq 4$. If $f(\lambda) = 0$ and $|\lambda| = \theta^{-1}$ then $\lambda = \theta^{-1}$.
\end{prop}
\begin{proof}
We recall some notation. We write $D=\{\!\{d_1,\ldots,d_g\}\!\}$ for the multi-set of degrees of the generators $x_1,\ldots,x_g$.
The degree of $x_i$ is denoted by $d_i\ge 1$ and there is an integer $d \ge 2$ such that $d_i+d_{g-i+1}=d$ for all $i=1,\ldots,g$. 
We write $n_j$ for the number of $x_i$ having degree $j$. We have also made the assumption that $\gcd(D)=1$. 

Suppose $d=2$. Then $\deg(x_i)=1$ for all $i$ so $f(t)=t^2-gt+1$ with $g \ge 4$. The result is true because the 
only roots of $f$ are $\theta$  and $\theta^{-1}$. 

For the remainder of the proof we assume that $d \geq 3$. 
Therefore $\deg(x_i)\ge 2$ for some $i$. Let $\omega=(\lambda\theta)^{-1}=\cos(\alpha)+i\sin(\alpha)$ and define
\begin{equation}
\label{eq.real.part2}
Z:=\cos(d\alpha)-1 + \sum_{j=1}^{d-1}n_j\theta^{-j} (1-\cos((d-j)\alpha)).
\end{equation}
Since $f(\theta^{-1})=f(\omega^{-1}\theta^{-1})=0$,
\begin{align*}
0 & = \theta^{-d} + 1 - \sum_{j=1}^{d-1}n_j\theta^{-j}  
\\
&= (\omega\theta)^{-d} + 1 -\sum_{j=1}^{d-1}n_j(\omega\theta)^{-j}
\\
& = \theta^{-d}+\omega^{d} -\sum_{j=1}^{d-1}n_j\theta^{-j}\omega^{d-j}. 
\end{align*}
Therefore 
\begin{equation}
\label{eq.real.part1}
0=\omega^d - 1 + \sum_{j=1}^{d-1}n_j\theta^{-j}(1-\omega^{d-j}).
\end{equation}
The real part of (\ref{eq.real.part1}) is $Z$, so $Z=0$. 

We will use the fact that $Z=0$ to show that $\cos(j\alpha)=1$ for all  $j \in D$; i.e., $\omega^j=1$ for all  $j \in D$. 
But $\gcd D = 1$ so it will then follow that $\omega=1$, whence $\lambda=\theta^{-1}$. 
The proof of the proposition will then be complete.  

\underline{Claim:}
One of the following is true:
 \begin{enumerate}
  \item[(a)]
 there is $q \in D$ such that $n_q \ge 2$ and $q < d/2$;
 \item[(b)]
 (a) does not occur, $d$ is even, $n_{d/2} \ge 2$, and $n_q \ne 0$ for some $q < d/2$;
 \item[(c)] 
neither (a) nor (b) occurs, and 
there are four distinct elements $q,r,d-q,d-r \in D$ such that $q,r<d/2$. 
\end{enumerate}

\underline{Proof of Claim:}
Suppose there is an integer $r \ne d/2$ such that two different $x_i$'s have degree $r$.
Then there are two different $x_i$'s having degree $d-r$. Since either $r$ or $d-r$ is $<d/2$, (a) holds in this case.
Conversely, if (a) holds there is an integer $r \ne d/2$ such that two different $x_i$'s have degree $r$.

Suppose (a) does not occur. Then either no two $x_i$'s have the same degree or $d$ is even and there are two different 
$x_i$'s having degree $d/2$. In the latter case, $n_{d/2} \ge 2$ and, since $\gcd(D)=1$, there is some $x_i$ whose degree is not
$d/2$, and hence an $x_j$ whose degree is $d-d_i$; either $d_i$ or $d-d_i$ is $<d/2$ so (b) holds. In the former case, since $g \ge 4$, 
(c) holds. $\lozenge$

In the three cases in the claim we make the following definitions:
\begin{enumerate}
  \item[(a)] 
  $C=\{q,d-q\}$, 
  \newline
  $n'_q=n'_{d-q}=2$, and 
  \newline
$
X= 3+\cos(d\alpha) - 2 \cos(q\alpha) - 2\cos((d-q)\alpha)
$
  \item[(b)]  
  $C=\{q,d-q,d/2\}$, 
  \newline
  $n'_q=n'_{d-q}=1$, $n'_{d/2}=2$,  and 
\newline
$X=3+\cos(d\alpha) -  \cos(q\alpha) - 2\cos(d\alpha/2) - \cos((d-q)\alpha) 
$
  \item[(c)]
  $C=\{q,r,d-q,d-r\}$, 
  \newline
  $n'_q=n'_r=n'_{d-q}=n'_{d-r}=1$, and 
  \newline
$
X:=3+\cos(d\alpha) -  \cos(q\alpha) - \cos(r\alpha) - \cos((d-q)\alpha) - \cos((d-r)\alpha)
$.
\end{enumerate}
Let $q$ and $r$ be as in (a), (b), and (c). Define $\eta=d\alpha/2$, $\psi=
(d/2-q)\alpha$, and $\phi=(d/2-r)\alpha$. In case (a), Lemma \ref{lem.highschool4}(3) implies $X \ge 0$.
In case (b), Lemma \ref{lem.highschool4}(2) implies $X \ge 0$.
In case (c), Lemma \ref{lem.highschool4}(1) implies $X \ge 0$.

In each of the three cases, 
$$
X=\cos(d\alpha)-1+\sum_{j\in C}n'_j(1-\cos((d-j)\alpha)).
$$
Therefore, $-X=Z-X=$  
$$
  \sum_{j \in D-C}n_j\theta^{-j} (1-\cos((d-j)\alpha)) + \sum_{j \in C}(n_j\theta^{-j}-n'_j) (1-\cos((d-j)\alpha)).
$$
However, 
\begin{itemize}
  \item 
  $n_j\theta^{-j}>0$ for all $j \in D-C$,
  \item 
  $n_j\theta^{-j}-n'_j>0$ for all $j \in C$, and
  \item 
  $1-\cos((d-j)\alpha)$ is always $\ge 0$,
\end{itemize} 
so $-X \ge 0$ and $-X=0$ if and only if $\cos(j\alpha)=1$ for all $j \in D$. 
Since both $-X$ and $X$ are $\ge0$ we conclude that $X=0$, whence $\cos(j\alpha)=1$ for all $j \in D$. The proof is now complete.
\end{proof}

\begin{lem}
Suppose $g \geq 4$. If $\lambda$ is a root of $f$ other than $\theta$, then $\theta < |\lambda|$.
\end{lem}
\begin{proof}
Let $\rho$ be a  root of $f$ having minimal modulus. Because $f(t)$ is the reciprocal of a Hilbert series, the power series 
$$
\frac{1}{f(t)} = \sum_{i = 0}^\infty a_i t^i
$$
has non-negative coefficients and a finite radius of convergence $|\rho|$. By Pringsheim's Theorem \cite[Theorem IV.6]{ancomb}, $|\rho|$ is a singularity of $H_A(t)$, i.e. a root of $f$. Since $\theta$ is the smallest positive root of $f$, $|\rho| = \theta$, so $\theta \leq |\lambda|$.

Suppose $|\lambda| = \theta$. Since $f$ is reciprocal, $\lambda^{-1}$ is a root of $f$, and $|\lambda^{-1}| = \theta^{-1}$. By Proposition \ref{equalmodulus}, $\lambda^{-1} = \theta^{-1}$, so $\lambda = \theta$, a contradiction. The result follows.
\end{proof}

\section{Examples}

\subsection{}
When $A$ is generated by $g \ge 3$ elements of degree one, $f(t)$ is the irreducible polynomial 
$1 - gt + t^2$ so
$$
K_0(\qgr(A)) \;  \cong  \; \bZ\left[\frac{g - \sqrt{g^2 - 4}}{2}\right] \; \subset \; \bR
$$
as ordered abelian groups.

\subsection{Non-irreducible $f$}
Suppose $g = 4$, $d_1 = d_2 = 1$ and $d_3 = d_4 = 2$. Then $f(t)=  1 - 2t - 2t^2 + t^3 = (1 + t)(1 - 3t + t^2)$ and 
$\theta = \frac{1}{2}(3 - \sqrt{5})$. 
The map
$$
 \frac{\bZ[t,t^{-1}]}{(f)} \to \bZ \oplus \bZ[\theta], \quad \ol{p} \mapsto (p(-1), p(\theta))
$$
is an isomorphism of abelian groups. The image of the positive cone under that isomorphism $K_0(\qgr(A)) \longrightarrow  \bZ \oplus \bZ[\theta]$ is $(\bZ \oplus \bZ[\theta]_{\geq 0}) \cup \set{0}$.

\end{document}